\documentclass[letterpaper, 10 pt, oneside]{article}
\usepackage{amsmath,amssymb, amsthm} 
\usepackage{cmbright}
\usepackage[T1]{fontenc}
\usepackage{setspace}
\usepackage[pdftex,%
bookmarksopen=false,
pdfpagemode=none 
]{hyperref}
\usepackage{color}
\usepackage{pdfsync}

 \makeatletter
\newtheorem*{rep@theorem}{\rep@title}
\newcommand{\newreptheorem}[2]{%
\newenvironment{rep#1}[1]{%
 \def\rep@title{#2 \ref{##1}}%
 \begin{rep@theorem}}%
 {\end{rep@theorem}}}
\makeatother

\newtheorem{theorem}{Theorem}[section]
\newreptheorem{theorem}{Theorem}
\newtheorem{lemma}[theorem]{Lemma}

\newtheorem{claim}[theorem]{Claim}
\newtheorem{conjecture}[theorem]{Conjecture}

\newtheorem{problem}[theorem]{Problem}

\begin{document}

\title{Extremal Problems for Subset Divisors}
\author{Tony Huynh\footnote{Department of Computer Science, University of Rome, Via Salaria 113, 00198 Rome, Italy  Email: \url{tony.bourbaki@gmail.com}}}
\date{}
\maketitle
\begin{abstract}
Let $A$ be a set of $n$ positive integers.  We say that a subset $B$ of $A$ is a \emph{divisor} of $A$, if the sum of the elements in $B$ divides the sum of the elements in $A$.  We are interested in the following extremal problem.  For each $n$, what is the maximum number of divisors a set of $n$ positive integers can have?  We determine this function exactly for all values of $n$.   Moreover, for each $n$ we characterize all sets that achieve the maximum.  We also prove results for the $k$-subset analogue of our problem.  For this variant, we determine the function exactly in the special case that $n=2k$.  We also characterize all sets that achieve this bound when $n=2k$.  

  \bigskip\noindent \textbf{Keywords:} extremal combinatorics; exact enumeration
\end{abstract}

\section{Introduction}

Let $A$ be a finite set of positive integers and let $B$ be a subset of $A$.  We say that $B$ is a \emph{divisor} of $A$, if the sum of the elements in $B$ divides the sum of the elements in $A$.  We are interested in the number of divisors a set of positive integers can have.  Toward that end, we let $d(A)$ be the number of divisors of $A$ and we let $d(n)$ be the maximum value of $d(A)$ over all sets $A$ of $n$ positive integers.  We also study the $k$-subset version of this problem.  That is, we define $d_k(A)$ to be the number of $k$-subset divisors of $A$ and  we let $d(k,n)$ be the maximum value of $d_k(A)$ over all sets $A$ of $n$ positive integers.

This work is motivated by Problem 1 of the 2011 International Mathematical Olympiad~\cite{IMO}, where it is asked to determine $d(2,4)$.

\begin{problem} \label{imo}
Determine $d(2,4)$.  Moreover, find all sets of four positive integers $A$ with exactly $d(2,4)$ 2-subset divisors. 
\end{problem}

We begin by presenting the solution to Problem~\ref{imo}.

\begin{lemma} \label{olympiad}
For all sets $A$ of four positive integers, $d_2(A) \leq 4$.  Moreover, $d_2(A)=4$ if and only if 
\[
A= \{a, 5a, 7a, 11a\} \text{ or } A=\{a, 11a, 19a, 29a\}
\]
for some $a \in \mathbb{N}$.
\end{lemma}

\begin{proof}
Let $A=\{a_1, a_2, a_3, a_4\}$ be a set of positive integers with $a_1 < a_2 < a_3 < a_4$.  We use $\sum A$ to denote the sum of the elements in $A$.  Note that $\frac{1}{2} \sum A < a_2+a_4 < a_3+a_4 < \sum A$.  Thus $d_2(A) \leq 4$ as neither $\{a_2, a_4\}$ nor $\{a_3, a_4\}$ can divide $A$.  

Suppose $d_2(A)=4$.  This implies that both $\{a_1,a_4\}$ and $\{a_2, a_3\}$ divide $A$, and hence $a_1+a_4=a_2+a_3$.  Since $\{a_1,a_2\}$ and $\{a_1,a_3\}$ also divide $A$,  $(a_1+a_2) | (a_3+a_4)$ and $(a_1+a_3) | (a_2+a_4)$.  
Therefore, there exist $2 \leq j < k$ such that
\begin{enumerate}
\item[$(i)$]
$a_1+a_4=a_3+a_2$,

\item[$(ii)$]
$j(a_1+a_3)=a_2+a_4$, and

\item[$(iii)$]
$k(a_1+a_2)=a_3+a_4$.
\end{enumerate}
Adding $(i)$ and $(ii)$, we obtain $(j+1)a_1+(j-1)a_3=2a_2$.   Since $a_3 > a_2$, it follows that $j=2$.  Substituting $j=2$ and taking $3(i)+2(ii)+(iii)$ we obtain $(k+7)a_1=(5-k)a_2$.  This implies $(5-k) >0$, and so $k \in \{3,4\}$.  By solving the systems corresponding to the values $k=3$ and $k=4$  we are lead to the respective solutions
\[
A=\{a, 5a, 7a, 11a\} \text{ and } A=\{a, 11a, 19a, 29a\}.
\]
It is easy to check that any set $A$ of the above form does indeed satisfy $d_2(A)=4$.  
\end{proof}

\section{Lower bounds for $d(n)$ and $d(k,n)$} \label{lower}

In this section, we give constructions for sets of positive integers with many divisors and many $k$-subset divisors.  In the next section we derive matching upper bounds for $d(n)$ and $d(n,2n)$ and hence these sets are optimal.  Moreover, in Section~\ref{uniqueness}, we will show that these are almost all the sets achieving the maximum values.

Recall that $d(n)$ (respectively, $d(k,n)$) is the maximum number of divisors (respectively, $k$-subset divisors) a set of $n$ positive integers can have.  By convention, the sum of the elements in the empty set is zero, and so the empty set does not divide any set (except itself).  
 
 \begin{lemma} \label{lower1}
 For all $n \geq 1$, $d(n) \geq 2^{n-1}$.  
 \end{lemma}

\begin{proof}
The lemma clearly holds if $n=1$.  Thus, assume $n \geq 2$ and let $A'$ be any set of $n-1$ positive integers.  We show that we can choose an element $a$ such that
$A' \cup \{a\}$ has $2^{n-1}$ divisors.  Let 
\[
S:=\{s \in \mathbb{N}: s=\sum B \text{ for some non-empty $B \subseteq A'$} \}.
\]
Let $\ell$ be the least common multiple of the elements in $S$, and let $\ell'$ be a multiple of $\ell$ such that $\ell' - \sum A' \notin A$.  Set $a:= \ell' - \sum A'$ and  consider $A:=A' \cup \{a\}$.  Note that $\sum A= \ell'$.  Therefore, every non-empty subset of $A'$ divides $A$.  Also, $A$ divides $A$.   Thus $d(A) \geq 2^{n-1}$, as required.  
\end{proof}

A similar construction also gives lower bounds for $d(k,n)$.

\begin{lemma} \label{lower2}
For all $k,n \geq 1$, $d(k,n) \geq \binom{n-1}{k}$.  
\end{lemma}

\begin{proof}
Again, the lemma clearly holds for $n=1$.  So, for $n \geq 2$ arbitrarily choose a set $A'$ of $n-1$ positive integers and let 
\[
S:=\{ s \in \mathbb{N} : s=\sum B \text{ for some $B \subseteq A'$ with $|B|=k$} \}.
\]
The rest of the proof is identical to the proof of the previous lemma.  That is, we construct $a$ such that all $k$-subsets of $A'$ divide $A' \cup \{a\}$. 
\end{proof}

We point out that the same technique shows that the corresponding minimization problems for $d(A)$ and $d_k(A)$ are easy.  Namely, define $A$ to be \emph{prime}
if the only divisor of $A$ is $A$ itself.  

\begin{claim}
For each $n \in \mathbb{N}$, there exists infinitely many prime $n$-sets of integers.  
\end{claim}

\begin{proof}
Arbitrary choose a set $A'$ of $n-1$ positive integers, with $1 \notin A'$.    Choose a prime number $p$ such that $p \geq 2\sum A'$.  Finish by setting $a:=p-\sum A'$ and $A:=A' \cup \{a\}$.  
\end{proof}

\section{Upper bounds for $d(n)$ and $d(n,2n)$}

Let $A$ be a set of positive integers.   We say that a subset $B$ of $A$ is a \emph{halving set} if $\sum B = \frac{1}{2} \sum A$.  Evidently, $B$ is a halving set if and only if $A \setminus B$ is a halving set.  The next lemma is also obvious, but quite useful.  

\begin{lemma} \label{half}
If $B$ and $C$ are distinct halving sets, then $| B \triangle C | > 2$.  
\end{lemma}

A \emph{separation} of $A$ is a pair $\{B,C\}$, where $B$ and $C$ are disjoint subsets of $A$ with $B \cup C=A$.  Note that $\{B,C\}=\{C,B\}$.    A \emph{strong separation} is a separation $\{B,C\}$ where $|B|=|C|$.  We say that $\{B,C\}$ is \emph{barren} if neither $B$ nor $C$ divides $A$, \emph{neutral} if exactly one of $B$ or $C$ divides $A$, and \emph{abundant} if both $B$ and $C$ divide $A$.  Note that $\{B,C\}$ is an abundant separation if and only if $B$ and $C$ are both halving sets.   

Thus, one approach to obtain upper bounds for $d(n)$ (respectively, $d(n,2n)$) is to bound the number of abundant separations (respectively, abundant strong separations) of $A$.  

\begin{lemma} \label{abundant1}
Let $A$ be a set of $n$ positive integers.  Then $d(A) \leq 2^{n-1}+h$, where $h$ is the number of abundant separations of $A$.
\end{lemma}

\begin{proof}
Partition $2^A$ into pairs $\{B, A \setminus B\}$.  There are $2^{n-1}$ such separations.  Finish by observing that a separation contributes 0 to $d(A)$ if it barren, 1 to $d(A)$ if it is neutral, and 2 to $d(A)$ if it is abundant.
\end{proof}

Similarly, we have the following lemma.

\begin{lemma} \label{abundant2}
Let $A$ be a set of $2n$ positive integers.  Then $d_n(A) \leq \frac{1}{2}\binom{2n}{n}+h$, where $h$ is the number of abundant strong separations of $A$.
\end{lemma}

Note that these bounds match the lower bounds from the previous section if $h=0$.  
However, it is possible for a set to have many halving sets.  For example, consider $A=\{1, \dots, 4\ell\}$.  A theorem of Stanley~\cite{lefschetz} shows
that this example is in fact worst possible. 

Fortunately, we are able to determine $d(n)$ and $d(n,2n)$ using a different approach.  

We first handle $d(n)$ by showing that the bound from Lemma~\ref{lower1} is best possible for almost all values of $n$.

 \begin{lemma} \label{tight2}
 For all $n \geq 4$, $d(n)=2^{n-1}$.  
 \end{lemma}
 
 \begin{proof}
Let $n \geq 4$ and $A$ be a set of $n$ positive integers.  By Lemma~\ref{lower1} it suffices to show $d(A) \leq 2^ {n-1}$. If no separations of $A$ are abundant, then we are done by Lemma~\ref{abundant1}.  So we may assume that $A$ contains an abundant separation.  We proceed by defining an injection $\phi$ from the set of abundant separations to the set of barren separations.  Let $\{B,C\}$ be an abundant separation.  We may assume that $\min A \in B$.  Define $\phi (\{B,C\})$ to be $\{B \setminus \min A, C \cup \min A\}$.  First note that $\phi$ is injective.  Secondly, if $\min A < \frac{1}{6} \sum A$, then $\frac{1}{3}\sum A < \sum (B \setminus \min A) < \frac{1}{2} \sum A    $.  Thus, if $\min A < \frac{1}{6} \sum A$, then $\phi$ maps abundant separations to barren separations.  So we are done unless $\min A \geq \frac{1}{6} \sum A$.
 
Observe that if $A$ contains a halving set $H$ of size at least 3, then $\min A \leq \min H < \frac{1}{6} \sum A$.   Therefore, we are done unless $n=4$.  Let $A:=\{a_1, \dots, a_4\}$ with $a_1 < a_2 < a_3 < a_4$.  Since there are no halving sets of $A$ of size 3, it follows that $\{\{a_1,a_4\}, \{a_2,a_3\} \}$ is the unique abundant separation of $A$.  Now, since $a_1 \geq \frac{1}{6} \sum A$ it follows that $\frac{1}{3} \sum A < a_1 + a_2 < \frac{1}{2} \sum A$.  Thus, $\{ \{a_1, a_2\}, \{a_3, a_4\} \}$ is a barren separation, so we are done by defining $\phi (\{\{a_1,a_4\}, \{a_2,a_3\} \}):=\{ \{a_1, a_2\}, \{a_3, a_4\} \}$.
\end{proof} 

 It is easy to determine the small values of $d(n)$ by hand.  We omit the details. 
 
 \begin{lemma} \label{easy}
We have $d(1)=1, d(2)=2$, and $d(3)=5$.   If $|A|=3$, then $d(A)=5$ if and only if $A=\{a, 2a, 3a\}$ for some $a \in \mathbb{N}$.  
 \end{lemma}

We now show that for $d(n, 2n)$, the lower bound from Lemma~\ref{lower2} is also best possible for $n \geq 3$.

\begin{lemma} \label{tight}
For all $n \geq 3$, $d(n, 2n) = \frac{1}{2} \binom{2n}{n}$.
\end{lemma}

\begin{proof}
Let $A$ be a set of $2n$ positive integers, with $n \geq 3$.  By Lemma~\ref{lower2}, it suffices to show that $d_n(A) \leq \frac{1}{2} \binom{2n}{n}$.  First observe that if $A$ does not contain any abundant strong separations, then we are done by Lemma~\ref{abundant2}.  

So, we may assume that $A$ contains an abundant strong separation.  In this case, we proceed by defining an injection $\phi$ from the family of abundant strong separations to the family of barren strong separations. 
Let $\{B,C\}$ be an abundant strong separation.   We define 
\[
\phi(\{B,C\}):=\{(B \setminus \min B) \cup \min C, (C \setminus \min C) \cup \min B\}.
\]  

First note that if $\phi(\{B_1, C_1\})=\phi(\{B_2, C_2\})$ for $\{B_1, C_1\} \neq \{B_2, C_2\}$, then by relabelling we may assume that $|B_1 \cap B_2|=n-1$.  However, this contradicts Lemma~\ref{half}.  So $\phi$ is indeed an injection.  We finish the proof by showing that $\phi$ maps abundant separations to barren separations.  We may assume that $\min B < \min C$.  Let $B':=(B \setminus \min B) \cup \min C$ and $C':=(C \setminus \min C) \cup \min B$.   Clearly, $B'$ does not divide $A$.  Also, as $\sum B = \sum C = \frac{1}{2} \sum A$, both $\min C$ and $\min B$ are strictly less than $\frac{1}{2n} \sum A$.   Therefore 
\[
(\frac{1}{2}-\frac{1}{2n}) \sum A < \sum C' < \frac{1}{2} \sum A.
\]
Since $n \geq 3$, $(\frac{1}{2}-\frac{1}{2n}) \geq \frac{1}{3}$.  Thus, $C'$ also does not divide $A$.  
\end{proof}
 
 
 \section{Characterizing all extremal sets} \label{uniqueness}

 We now characterize all subsets of integers that achieve the bounds in Lemma~\ref{tight2} and Lemma~\ref{tight}.   
 Let $A:=\{a_1, \dots, a_n\}$ be a set of $n$ positive integers with $a_1 < \dots < a_n$.  We say that $A$ is an \emph{anti-pencil} if the set of divisors of $A$ consists of 
 all non-empty subsets of $A \setminus \{a_n\}$ together with $A$ itself.  
Similarly, $A$ is a \emph{$k$-anti-pencil} if the set of $k$-subset divisors of $A$ is the set of all $k$-subsets of $A \setminus \{a_n\}$.  
Observe that the constructions in Section~\ref{lower} completely describe the set of all anti-pencils and the set of all $k$-anti-pencils.  

We will need the following two simple observations to aid with the case analysis.

\begin{lemma} \label{diophantine}
If $k,\ell$, and $m$ are positive integers such that $\frac{1}{k}+\frac{1}{\ell}=\frac{1}{m}$, then $k+\ell$ divides $k\ell$.  
\end{lemma}

\begin{lemma} \label{diophantine2}
If $k< \ell$ are positive integers such that $\frac{1}{k}+\frac{1}{\ell}=\frac{1}{2}$, then $k=3$ and $\ell=6$.  
\end{lemma}

Here is our first characterization.  

\begin{lemma} \label{EKR2}
For $n \geq 5$, a set of $n$ positive integers has exactly $2^{n-1}$ divisors if and only if it is an anti-pencil.  
A set of four positive integers has 8 divisors if and only if it is an anti-pencil or of the form $\{a,2a,3a,6a\}$ for some $a \in \mathbb{N}$.  
\end{lemma} 
 
 \begin{proof}
 One direction is obvious.  For the other direction, let $n \geq 4$ and $A:=\{a_1, \dots, a_n\}$  have exactly $2^{n-1}$ divisors.  We may assume that $a_n \leq \frac{1}{2} \sum A$, else $A$ is an anti-pencil and we are done.    
 
 We claim that $a_2 \geq \frac{1}{6} \sum A$.  Suppose not.  Given an abundant separation of $A$, let $\phi_1$ be the map which moves $a_1$ across  the separation and let $\phi_2$ be the map which moves $a_2$ across the separation.  Since $a_1$ and $a_2$ are both less than $\frac{1}{6} \sum A$, we again have that $\phi_1$ and $\phi_2$ are injective maps from the set of abundant separations to the set of barren separations.  Moveover, by Lemma~\ref{half}, the images of $\phi_1$ and $\phi_2$ are disjoint.  Therefore, $A$ has more barren separations than abundant separations, which is a contradiction.  
 
 We next claim that $A$ does not contain any abundant separations or $A=\{a,2a,3a,6a\}$ for some $a \in \mathbb{N}$. Suppose $\{B,C\}$ is an abundant separation.  If $\max \{|B|,|C|\} \geq 4$ or $\min \{|B|,|C|\} \geq 3$, then $a_2 < \frac{1}{6} \sum A$; a contradiction.  In particular, this implies $n \in \{4,5\}$.  
 
 We first handle the case $n=5$.  Let $B:=\{b_1, b_2\}$ and $C:=\{c_1, c_2,c_3\}$ with $b_1<b_2$ and $c_1<c_2<c_3$.  By Lemma~\ref{half}, $\{B,C\}$ is the unique abundant separation of $A$.  Now, if $a_1=b_1$, then $a_2 \leq c_1 < \frac{1}{6} \sum A$; a contradiction.  Thus, $a_1=c_1$.  It follows that $\{\{b_1,b_2,c_1\}, \{c_2, c_3\} \}$ is the unique barren separation of $A$.  In particular $\{a_i\}$ divides $A$ for all $i \in [5]$.   Therefore, there exist positive integers $m_1>\dots>m_5$ such that $m_ia_i= \sum A$.  Since $m_2 \leq 6$ and $m_5 \geq 3$ we must have $m_2=6, m_3=5, m_4=4$, and $m_5=3$.  Now choose a 2-subset $A'$ of $\{a_2, a_3, a_4\}$ such that $A' \neq \{c_2, c_3\}$.  Since $\sum A' < \frac{1}{2} \sum A$ and $A' \neq \{c_2,c_3\}$, it follows that $A'$ divides $A$.  However, this is a contradiction, since the equation $\frac{1}{k}+\frac{1}{\ell}=\frac{1}{m}$ has no positive integer solutions for $\{k,\ell\} \subset \{4,5,6\}$ by Lemma~\ref{diophantine}.  
 
We thus have $n=4$.  Again by Lemma~\ref{half}, $\{B,C\}$ is the unique abundant separation of $A$.  Thus, there is a unique barren separation $\{B',C'\}$ of $A$. 
 First suppose $|B|=|C|=2$.  By relabelling if necessary,  $B=\{a_1,a_4\}$ and $C=\{a_2,a_3\}$.  Since $\{B',C'\}$ is the unique barren separation of $A$, at least three of $\{a_1\}, \{a_2\}, \{a_3\}$ or $\{a_4\}$ divide $A$.  By the pigeonhole principle, both members of $B$ divide $A$ or both members of $C$ divide $A$.  By Lemma~\ref{diophantine2}, we either have $6a_1=\sum A=3a_4$ or $6a_2=\sum A=3a_3$.   In the second case, swapping $a_1$ and $a_2$ or swapping $a_3$ and $a_4$ in $\{B,C\}$ both yield barren separations, which contradicts the uniqueness of $\{B',C'\}$.  The first case is also impossible as any single swap of $\{B,C\}$ yields a barren separation.  Therefore, we may assume $B=\{a_4\}$ and $C=\{a_1,a_2,a_3\}$.  Now consider the sets
 \[
 \{a_2\}, \{a_3\}, \{a_1,a_2\}, \{a_1, a_3\}, \{a_2, a_3\}.
 \]
 Since $A$ has exactly one barren separation, at least four of these sets divide $A$.  Let $m_1 \leq m_2 \leq m_3 \leq m_4$ be the multiples that appear for these four divisors $D_1, \dots, D_4$.   
 Note that $m_4 \leq 6$ since $a_2 \geq \frac{1}{6} \sum A$.  Suppose these multiples are all distinct.  In this case, it follows that $m_1=3, m_2=4, m_3=5, m_4=6$.  Since $\frac{1}{3} + \frac{1}{6}=\frac{1}{2}$, $D_1$ and $D_4$ must partition $\{a_1, a_2, a_3\}$.   It cannot be that $\{D_1,D_4\}=\{\{a_3\}, \{a_1, a_2\}\}$ since then neither $\{a_1, a_3\}$ nor $\{a_2, a_3\}$ divides $A$.  Thus, $D_1=\{a_1, a_3\}$ and $D_4=\{a_2\}$.  But now, $D_2$ and $D_3$ also partition $\{a_1, a_2,a_3\}$.  This is a contradiction since $\frac{1}{4}+\frac{1}{5} \neq \frac{1}{2}$.  Therefore, $m_i=m_j$ for some $i \neq j$.  This is only possible if $D_i$ and $D_j$ partition $\{a_1,a_2,a_3\}$.  Thus, $m_i=m_j=4$, $D_i=\{a_1,a_2\}$, and $D_j=\{a_3\}$.  Observe that at least one of $\{a_1,a_3\}$ or $\{a_2,a_3\}$ divides $A$, else $A$ has two barren separations.  Thus, $a_1+a_3=\frac{1}{3} \sum A$ or $a_2+a_3=\frac{1}{3} \sum A$.  So, either $a_1=\frac{1}{12} \sum A$ or $a_2=\frac{1}{12} \sum A$. It must be that $a_1=\frac{1}{12} \sum A$, since $a_2 \geq \frac{1}{6} \sum A$.  Finally, $a_2=\frac{1}{6} \sum A$, since $a_1+a_2= \frac{1}{4} \sum A$.  Thus, $A:=\{a,2a,3a,6a\}$ for some $a \in \mathbb{N}$, as required.
 
We may hence assume that every separation of $A$ is a neutral separation.  Thus, $B$ divides $A$ if and only if $\sum B < \frac{1}{2} \sum A$.  In particular, $\{a_n\}$ divides $A$ and so $a_n \leq \frac{1}{3} \sum A$.  Let $M$ be a maximal set (under inclusion) among all subsets of $A$ containing $a_n$ and with sum at most $\frac{1}{2}\sum A$.  Note that $\{a_n\}$ is a candidate for $M$, so $M$ exists.  Since $\sum M < \frac{1}{2} \sum A$, it follows that $M$ divides $A$ and hence $\sum M \leq \frac{1}{3} \sum A$.  Choose $a \notin M$ and consider $M \cup \{a\}$.  By choice of $M$ we have $\sum (M \cup \{a\}) > \frac{1}{2} \sum A$.  Thus, $A \setminus (M \cup \{a\}) \leq \frac{1}{3} \sum A$, since  $A \setminus (M \cup \{a\})$ divides $A$.  But now, $a_n > a \geq \frac{1}{3} \sum A$, which is a contradiction.  
 \end{proof}
 
 Combining Lemma~\ref{EKR2} with Lemma~\ref{easy} we have the following summary.
 
 \begin{theorem}
 For all $n \neq 3$, $d(n)=2^{n-1}$ and the sets that achieve this bound are precisely the anti-pencils or $A:=\{a,2a,3a,6a\}$ for some $a \in \mathbb{N}$.  For $n=3$, $d(3)=5$ and the sets that achieve this bound are
$\{a,2a,3a\}$ for some $a \in \mathbb{N}$.  
 \end{theorem}
 
 We now present the end of the story for $d(n, 2n)$ as well.

\begin{lemma} \label{EKR1}
Let $n \geq 3$ and let $A$ be a set of $2n$ positive integers.  If $A$ has exactly $\frac{1}{2} \binom{2n}{n}$ divisors of size $n$, then $A$ is an $n$-anti-pencil.
\end{lemma}

\begin{proof}
Let $n \geq 3$ and let $A$ be a set of $2n$ positive integers with  $\frac{1}{2} \binom{2n}{n}$ divisors of size $n$.   Suppose $A:=\{a_1, \dots, a_{2n} \}$ with $a_1 < \dots < a_{2n}$.  

We claim that $A$ does not contain any abundant strong separations.  Suppose not.  We first suppose $n=3$.  By Lemma~\ref{half}, $A$ has a unique abundant strong separation $\{\{b_1, b_2, b_3\}, \{c_1, c_2, c_3\}\}$ with elements labelled in increasing order.  Note that $b_1 < \frac{1}{6} \sum A$ and $c_1 < \frac{1}{6} \sum A$.  Thus, swapping $b_1$ and $c_1$ yields a barren separation.  If $b_2 \leq \frac{1}{6} \sum A$, then swapping $b_2$ and $c_1$ yields another barren separation; a contradiction.  Thus, $b_2 > \frac{1}{6} \sum A$.  On the other hand, $b_2 < \frac{1}{4} \sum A$ since $b_2 < b_3$.  By symmetry, we also have $\frac{1}{6}\sum A < c_2 < \frac{1}{4} \sum A$.  Thus, swapping $b_2$ and $c_2$ yields another barren separation.  So we may assume $n \geq 4$.  
Let $\{B,C\}$ be an abundant separation of $A$ with $a_1 \in B$.  Let $c_1$ and $c_2$ be the two smallest elements of $C$.  Define 
\[
\phi_1 (\{B,C\}):= \{ (B \setminus \{a_1\}) \cup \{c_1\} , (C \setminus \{c_1\}) \cup \{a_1\} \} \text{  and}
\]
\[ 
\phi_2 (\{B,C\}):= \{ (B \setminus \{a_1\}) \cup \{c_2\} , (C \setminus \{c_2\}) \cup \{a_1\}\}.
\]
Since $n \geq 4$, it follows that $a_1, c_1$ and $c_2$ are each less than $\frac{1}{6} \sum A$.  Therefore, both $\phi_1$ and $\phi_2$ are injective maps from the set of abundant strong separations to the set of barren strong separations.  Furthermore, by Lemma~\ref{half}, the images of $\phi_1$ and $\phi_2$ are disjoint.  Therefore $A$ contains more barren strong separations than abundant strong separations; a contradiction.   Thus, $A$ has no abundant strong separations as claimed.  

It follows that every strong separation of $A$ must be neutral.  Thus, an $n$-subset $B$ divides $A$ if and only if $\sum B < \frac{1}{2} \sum A$.  Consider $M:=\{a_1, \dots, a_{n-1}, a_{2n}\}$.  Note that $\sum M \leq \frac{1}{2} \sum A$, else $A$ is an $n$-anti-pencil and we are done.  Hence, in fact $\sum M \leq \frac{1}{3} \sum A$. Now let $B_1, \dots, B_{\ell}$ be a sequence of $n$-subsets of $A$ such that $B_1=M$, $B_{\ell}=\{a_{n+1}, \dots, a_{2n}\}$, and for each $1<j\leq \ell$, $B_j$ is obtained from $B_{j-1}$ by replacing some element of $B_{j-1}$ by a larger element not in $B_{j-1}$. Let $k$ be the first index such that $\sum B_k > \frac{1}{2} \sum A$, and let $B_k=B_{k-1} \triangle \{c,d\}$ with $c < d$.   Since $B_{j-1}$ and $A \setminus B_j$ both have sum less than $\frac{1}{2} \sum A$, they both must divide $A$.  Therefore, $B_{j-1} \leq \frac{1}{3} \sum A$ and $A \setminus B_j \leq \frac{1}{3} \sum A$.  It follows that $\sum A - d+c \leq \frac{2}{3} \sum A$.  
Thus, $a_{2n}>d>\frac{1}{3} \sum A$, which is a contradiction since $a_{2n} \in M$.  
\end{proof}

Combining Lemma~\ref{EKR1} with Lemma~\ref{olympiad} we have the following summary.

 \begin{theorem}
 For all $n \neq 2$, $d(n,2n)=\frac{1}{2} \binom{2n}{n}$ and the sets of $2n$ positive integers that achieve this bound are precisely the $n$-anti-pencils.  For $n=2$, $d(2,4)=4$ and the sets that achieve this bound are
$A= \{a, 5a, 7a, 11a\}$ or $A=\{a, 11a, 19a, 29a\}$
for some $a \in \mathbb{N}$.
 \end{theorem}

\section{Continuous analogues and open problems}
 
 If one considers subsets of \textit{real} numbers instead of natural numbers, then the question of divisibility no longer makes sense.  In this context, it is natural to instead ask 
 about subsets  with \textit{non-negative} sum.  
 
 Let $A$ be a set of $n$ real numbers.   Define $\mu(A)$ (respectively $\mu_k(A)$) to be the number of subsets (respectively $k$-subsets) $B$ of $A$ such that $\sum B \geq 0$.  
 We then define $\mu_{\min} (n)$ (respectively $\mu_{\max} (n)$) to be the minimum (respectively maximum) of $\mu(A)$ over all sets $A$ of $n$ real numbers with $\sum A=0$.  Similarly, we 
 define
 $\mu_{\min}(k,n)$ (respectively $\mu_{\max}(k,n)$) to be the minimum (respectively maximum) of $\mu_k(A)$ over all sets $A$ of $n$ real numbers with
 $\sum A=0$.  
 
 It is easy to check that $\mu_{\min}(n)=2^{n-1}+1$ for all $n \geq 1$, and that $A= \{1, \dots, n-1\} \cup \{\frac{-(n)(n-1)}{2}\}$ achieves this bound (recall that $\sum \emptyset = 0$).

On the other hand, the minimization problem for $\mu_k$ is non-trivial.  Indeed, the following nice conjecture of Manickam, Mikl{\'o}s, and Singhi
 asserts that $\mu_{\min}(k,n) \geq \binom{n-1}{k-1}$, for $n \geq 4k$.  
 
\begin{conjecture}[\cite{MMS1}, \cite{MMS2}]
If $n$ and $k$ are positive integers with $n \geq 4k$, and $A$ is a set of $n$ real numbers with $\sum A=0$, then the number of subsets of $A$ with
non-negative sum is at least $\binom{n-1}{k-1}$.   
\end{conjecture}

Note that by choosing $A$ with exactly one non-negative element, we do obtain $\binom{n-1}{k-1}$ non-negative $k$-subsets.  Such sets correspond
to the extremal examples in the Erd{\H{o}}s-Ko-Rado theorem~\cite{EKR}, so we will call them \textit{$k$-pencils}.   

The Manickam-Mikl{\'o}s-Singhi conjecture has recently received substantial attention.  We refer the reader to  Alon, Huang and Sudakov~\cite{MMS3}, Chowdhury~\cite{ameerah}, Frankl~\cite{frankl}, and Pokrovskiy~\cite{linearMMS}.  

We now discuss  $\mu_{\max}(n)$ and $\mu_{\max}(k,n)$, which can be viewed as continuous analogues of our functions $d(n)$ and $d(k,n)$.  

It is easy to see that maximizing $\mu(n)$ is equivalent to maximizing the number of subsets of $A$ whose sum is \textit{exactly} zero.  
For $n$ odd, this reduces to a conjecture of Erd{\H{o}}s and Moser, see~\cite{extremalNT1, extremalNT2}.  Using the Hard Lefschetz theorem~\cite{lefschetzbook} from 
algebraic geometry, Stanley~\cite{lefschetz, stanley2} solved a (generalization) of the Erd{\H{o}}s-Moser conjecture.

Thus, by~\cite[Corollary 5.1]{lefschetz}, we have the following summary for $\mu_{\max}(n)$.  
 
\begin{theorem}
For $n=2\ell$, $\mu_{\max}(n)$ is achieved by taking $A=\{-\ell, \dots, -1\} \cup \{1, \dots, \ell\}$.  For $n=2\ell+1$, $\mu_{\max}(n)$ is achieved by taking $A=\{-\ell, \dots \ell\}$.
\end{theorem}

As far as we know, determining $\mu_{\max}(k,n)$ is a wide open problem, although similar questions have been considered.
For example, one can define $\mu'_{\max}(k,n)$ to be the maximum of $\mu(A)$ over 
all sets $A$ of $n$ real numbers such that $\sum A <0$ and $\sum B < 0$ for all $B \subseteq A$ with $|B| > k$.     
Recently, Alon, Aydinian and Huang~\cite{maxnonnegative} proved that $\mu'_{\max}(k,n)= \binom{n-1}{k-1}+ \dots + \binom{n-1}{0}+1$, settling a question of
Tsukerman.

Note that by choosing exactly one element of $A$ to be negative, we 
have the bound $\mu_{\max}(k,n) \geq \binom{n-1}{k}$.  We overload terminology and call such a set a
\textit{$k$-anti-pencil}.  This construction is not optimal for $n=2k$, but in the range of the Manickam-Mikl{\'o}s-Singhi conjecture, we
conjecture that it is.

\begin{conjecture}
If $n \geq 4k$, then $\mu_{\max}(k,n)=\binom{n-1}{k}$.
\end{conjecture}

One can also attempt to characterize the extremal examples for $\mu_{\min}(k,n)$ and $\mu_{\max}(k,n)$.  For example, Chowdhury~\cite{ameerah} gives some
values of $k$ and $n$ for which the extremal examples for $\mu_{\min}(k,n)$ are necessarily $k$-pencils.
Thus, it would be quite interesting to determine for which $k$ and $n$, the extremal examples for minimizing $\mu(k,n)$ and maximizing $\mu(k,n)$ are necessarily `dual' to one another.  

\begin{problem}
Determine for which $k$ and $n$ the only extremal examples for $\mu_{\min}(k,n)$ and $\mu_{\max}(k,n)$ are  $k$-pencils and $k$-anti-pencils, respectively.
\end{problem}

We end by mentioning that determining $d(k,n)$ for $n \neq 2k$ is also an open problem.  Recall that our proof technique relies on the fact that $B$ and $A \setminus B$ are both possible divisors of $A$, 
which fails when $n \neq 2k$.  Nonetheless, we conjecture that for most values of $k$ and $n$, $d(k,n)=\binom{n-1}{k}$.  Note that this agrees with the conjectured value for $\mu_{\max}(k,n)$. 
However, in the divisibility setting, it is possible that $d(k,n)=\binom{n-1}{k}$ for all but finitely many values.  

\begin{conjecture}
For all but finitely values of $k$ and $n$, $d(k,n)=\binom{n-1}{k}$.  
\end{conjecture}

\subsection*{Acknowledgements}
We thank Alex Heinis for valuable discussions during the 2011 International Mathematical Olympiad in Amsterdam.  We also thank Ameera Chowdhury 
and Gjergji Zaimi for providing useful references.  
 
\bibliography{references}{}

\begin{thebibliography}{10}

\bibitem{IMO}
{International Mathematical Olympiad Problems}.
\newblock {\em \url{http://www.imo-official.org/problems.aspx}}.

\bibitem{maxnonnegative}
Noga Alon and Hao Huang.
\newblock Maximizing the number of nonnegative subsets.
\newblock {\em \url{http://arxiv.org/abs/1312.0248}}, 2013.

\bibitem{MMS3}
Noga Alon, Hao Huang, and Benny Sudakov.
\newblock Nonnegative {$k$}-sums, fractional covers, and probability of small
  deviations.
\newblock {\em J. Combin. Theory Ser. B}, 102(3):784--796, 2012.

\bibitem{ameerah}
Ameera Chowdhury.
\newblock A note on the {M}anickam-{M}ikl{\'o}s-{S}inghi conjecture.
\newblock {\em Eur. J. Comb.}, 35:131--140, 2014.

\bibitem{extremalNT2}
R.~C. Entringer.
\newblock Representation of {$m$} as {$\sum \sb{k=-n}\sp{n}\,\varepsilon
  \sb{k}k$}.
\newblock {\em Canad. Math. Bull.}, 11:289--293, 1968.

\bibitem{extremalNT1}
P.~Erd{\H{o}}s.
\newblock Extremal problems in number theory.
\newblock In {\em Proc. {S}ympos. {P}ure {M}ath., {V}ol. {VIII}}, pages
  181--189. Amer. Math. Soc., Providence, R.I., 1965.

\bibitem{EKR}
P.~Erd{\H{o}}s, Chao Ko, and R.~Rado.
\newblock Intersection theorems for systems of finite sets.
\newblock {\em Quart. J. Math. Oxford Ser. (2)}, 12:313--320, 1961.

\bibitem{frankl}
Peter Frankl.
\newblock On the number of nonnegative sums.
\newblock {\em J. Combin. Theory Ser. B}, 103(5):647--649, 2013.

\bibitem{lefschetzbook}
S.~Lefschetz.
\newblock {\em L'analysis Situs Et la G{\'e}om{\'e}trie Alg{\'e}brique}.
\newblock Collection de monographies sur la th{\'e}orie des fonctions.
  Gauthier-Villars et cie, 1924.

\bibitem{MMS1}
N.~Manickam and D.~Mikl{\'o}s.
\newblock On the number of nonnegative partial sums of a nonnegative sum.
\newblock In {\em Combinatorics ({E}ger, 1987)}, volume~52 of {\em Colloq.
  Math. Soc. J\'anos Bolyai}, pages 385--392. North-Holland, Amsterdam, 1988.

\bibitem{MMS2}
N.~Manickam and N.~M. Singhi.
\newblock First distribution invariants and {EKR} theorems.
\newblock {\em J. Combin. Theory Ser. A}, 48(1):91--103, 1988.

\bibitem{linearMMS}
Alexey Pokrovskiy.
\newblock A linear bound on the {M}anickam-{M}ikl{\'o}s-{S}inghi conjecture.
\newblock {\em \url{http://arxiv-web3.library.cornell.edu/abs/1308.2176}},
  2013.

\bibitem{lefschetz}
Richard~P. Stanley.
\newblock Weyl groups, the hard {L}efschetz theorem, and the {S}perner
  property.
\newblock {\em SIAM J. Algebraic Discrete Methods}, 1(2):168--184, 1980.

\bibitem{stanley2}
Richard~P. Stanley.
\newblock Some applications of algebra to combinatorics.
\newblock {\em Discrete Appl. Math.}, 34(1-3):241--277, 1991.
\newblock Combinatorics and theoretical computer science (Washington, DC,
  1989).

\end{thebibliography}
\bibliographystyle{plain}

\end{document}